\documentclass[11pt,reqno]{amsart}

\usepackage{amsmath,amssymb,amsthm,bbm,bm}
\usepackage{color}
\usepackage[colorlinks=true, allcolors=blue]{hyperref}
\usepackage{mathrsfs}
\usepackage{mathtools}

\usepackage[noabbrev,capitalize,nameinlink]{cleveref}
\crefname{equation}{}{}
\usepackage[noadjust]{cite}

\usepackage{fullpage}

\usepackage{graphics}
\usepackage{pifont}
\usepackage[T1]{fontenc}
\usepackage{tikz}
\usetikzlibrary{arrows.meta}
\usepackage{environ}
\usepackage{framed}
\usepackage{url}
\usepackage[linesnumbered,ruled,vlined]{algorithm2e}
\usepackage[noend]{algpseudocode}
\usepackage[labelfont=bf]{caption}
\usepackage{cite}
\usepackage{framed}
\usepackage[framemethod=tikz]{mdframed}
\usepackage{appendix}
\usepackage{graphicx}
\usepackage[textsize=tiny]{todonotes}
\usepackage{tcolorbox}
\usepackage{enumerate}
\allowdisplaybreaks[1]
\usepackage{enumerate}

\crefname{algocf}{Algorithm}{Algorithms}

\crefname{equation}{}{} %remove ``Equation''
\AtBeginEnvironment{appendices}{\crefalias{section}{appendix}} %appendices

\usepackage[final, color]{showkeys} %add in 'final' into parameter to remove showkeys
\colorlet{refkey}{orange!20}
\colorlet{labelkey}{blue!30}

\crefname{algocf}{Algorithm}{Algorithms}

\numberwithin{equation}{section}
\newtheorem{theorem}{Theorem}[section]
\newtheorem{proposition}[theorem]{Proposition}
\newtheorem{lemma}[theorem]{Lemma}

\crefname{claim}{Claim}{Claims}

\newtheorem{conjecture}[theorem]{Conjecture}
\newtheorem*{question*}{Question}

\theoremstyle{definition}
\newtheorem{definition}[theorem]{Definition}

\newtheorem*{definition*}{Definition}

\theoremstyle{remark}

\newcommand{\Mod}[1]{~(\mathrm{mod}~#1)}

\newcommand{\mb}{\mathbb}

\newcommand{\on}{\operatorname}

\newcommand{\C}{\mathbb C}
\newcommand{\E}{\mathbb E}
\newcommand{\F}{\mathbb F}

\newcommand{\R}{\mathbb R}
\newcommand{\U}{\mathbb U}
\newcommand{\Z}{\mathbb Z}

\DeclareMathOperator{\Poly}{Poly}

\renewcommand{\Pr}{\mathbb P}

\allowdisplaybreaks

\title{Non-classical polynomials and the inverse theorem}

\author[Berger]{Aaron Berger}
\author[Sah]{Ashwin Sah}
\author[Sawhney]{Mehtaab Sawhney}
\author[Tidor]{Jonathan Tidor}
\address{Department of Mathematics, Massachusetts Institute of Technology, Cambridge, MA 02139, USA}
\thanks{Berger, Sah, Sawhney, and Tidor were supported by NSF Graduate Research Fellowship Program DGE-1745302.}
\email{\{bergera,asah,msawhney,jtidor\}@mit.edu}

\begin{document}

\begin{abstract}
In this note we characterize when non-classical polynomials are necessary in the inverse theorem for the Gowers $U^k$-norm. We give a brief deduction of the fact that a bounded function on $\F_p^n$ with large $U^k$-norm must correlate with a classical polynomial when $k\le p+1$. To the best of our knowledge, this result is new for $k=p+1$ (when $p>2$). We then prove that non-classical polynomials are necessary in the inverse theorem for the Gowers $U^k$-norm over $\F_p^n$ for all $k\ge p+2$, completely characterizing when classical polynomials suffice.
\end{abstract}

\maketitle

\section{Introduction}\label{sec:intro}

The inverse theorem for the Gowers $U^k$-norm states that a bounded function $f\colon G\to\C$ has large $U^k$-norm if and only if $f$ correlates with a certain structured object. When $G=\Z/N\Z$, these structured objects are quite complicated and need the theory of nilsequences to describe. When $G=\F_p^n$, the situation is somewhat simpler. When $p>k$, a bounded function $f\colon\F_p^n\to\C$ has large $U^k$-norm if and only if $f$ has non-negligible correlation with a polynomial phase function, i.e., $e^{2\pi i P(x)/p}$ where $P\colon\F_p^n\to\F_p$ is a polynomial of degree at most $k-1$.

The situation when $p$ is small compared to $k$ is more delicate. Green and Tao \cite{GT09} and independently Lovett, Meshulam, and Samorodnitsky \cite{LMS11} showed that the corresponding conjecture is false for $k=4$ and $p=2$. In other words, there exist bounded functions $f\colon\F_2^n\to\C$ with large $U^4$-norm but with correlation $o_{n\to\infty}(1)$ with every cubic phase function. Tao and Ziegler \cite{TZ12} clarified this situation by proving that for all $k$ and $p$, a bounded function $f\colon\F_p^n\to\C$ has large $U^k$-norm if and only if $f$ has non-negligible correlation with a non-classical polynomial phase function, i.e., $e^{2\pi i P(x)}$ where $P\colon\F_p^n\to\R/\Z$ is a non-classical polynomial of degree at most $k-1$. (See \cref{sec:background} for the relevant definitions.)

A natural question which remains from the above discussion is to determine for which pairs $p,k$ does the $U^k$-inverse theorem over $\F_p^n$ hold with classical polynomials. In the positive direction, it is known due to Samorodnitsky \cite{Sam07} that the $U^3$-inverse theorem over $\F_2^n$ holds with classical polynomials. In the negative direction, Lovett, Meshulam, and Samorodnitsky \cite{LMS11} proved that the $U^{p^\ell}$-inverse theorem over $\F_p^n$ requires non-classical polynomials for all $p$ and $\ell\geq 2$. (A curious feature of this problem is that it is not monotone in $k$, e.g., the Lovett-Meshulam-Samorodnitsky result does not imply that non-classical polynomials are necessary in the $U^k$-inverse theorem for all $k\geq p^2$.)

In this paper we completely characterize when classical polynomials suffice in the statement of the inverse theorem. We first prove the inverse theorem for the Gowers $U^{p+1}$-norm with classical polynomials. This result is proved via a short deduction from the usual inverse theorem for the $U^{p+1}$-norm that involves non-classical polynomials.\footnote{See \cref{sec:background} for the definitions and notation used in the statement of these results.}
 
\begin{theorem}\label{thm:p+1-inverse}
Fix a prime $p$ and $\delta>0$. There exists $\epsilon>0$ such that the following holds. Let $V$ be a finite-dimensional $\mb{F}_p$-vector space. Given a function $f\colon V\to\mb{C}$ satisfying $\|f\|_\infty\le1$ and $\|f\|_{U^{p+1}}>\delta$, there exists a classical polynomial $P\in\on{Poly}_{\leqslant p}(V\to\mb{F}_p)$ such that
\[|\mb{E}_{x\in V} f(x)e_p(-P(x))|\ge\epsilon.\]
\end{theorem}

Second, we give an example showing that non-classical polynomials are necessary in the $U^k$-inverse theorem for all $k\geq p+2$.

\begin{theorem}\label{thm:main}
Fix a prime $p$ and an integer $k\geq p+2$. For all $n$, there exists a function $f\colon\F_p^n\to \C$ satisfying $\|f\|_{\infty}\leq 1$ and $\|f\|_{U^k}=1$ such that for all (classical) polynomials $P\in\on{Poly}_{\leqslant k-1}(\F_p^n\to\F_p)$,
\[|\mb{E}_{x\in \F_p^n} f(x)e_p(-P(x))|=o_{p,k;n\to\infty}(1).\]
\end{theorem}

Our example is fairly simple to write down. For $k\geq p+2$, we write $k-1=r+(p-1)\ell$ where $\ell\geq 1$ and $0<r<p$. Then our function is
\[f(x)=e^{{2\pi i}\frac{\sum_{i=1}^n|x_i|^r}{p^{\ell+1}}}\]
(where $|\cdot|\colon\F_p\to\{0,\ldots,p-1\}$ is the standard map). Note that this function $f$ is a non-classical polynomial phase function of degree $k-1$, so the content of this result is that it does not correlate with any classical polynomial phase functions of the same degree.

The $o(1)$ correlation in \cref{thm:main} is fairly bad -- the inverse of many iterated logarithms. This is due to our use of a Ramsey-theoretic argument inspired by a similar argument of Alon and Beigel. We conjecture that this bound on the correlation can be improved.

\begin{conjecture}\label{conj:exp-bounds}
Fix a prime $p$ and an integer $k\geq p+2$. For all $n$ there exist $f\colon\F_p^n\to\C$ satisfying $\|f\|_{\infty}\leq 1$ and $\|f\|_{U^k}\geq c_{p,k}>0$ such that for all (classical) polynomials $P\in\on{Poly}_{\leqslant k-1}(\F_p^n\to\F_p)$,
\[|\mb{E}_{x\in \F_p^n} f(x)e_p(-P(x))|\leq \exp(-\Omega_{p,k}(n)).\]
\end{conjecture}

In fact, we believe that this conjecture is true with the same functions that we use to prove \cref{thm:main}.

\medskip

\textbf{Structure of the paper:} In \cref{sec:background} we give the definition of the Gowers $U^k$-norm and of non-classical polynomials. In \cref{sec:p+1-inverse} we prove \cref{thm:p+1-inverse}. We prove \cref{thm:main} in the remainder of the paper. \cref{sec:symmetrization} develops the symmetrization tool that we use and \cref{sec:non-classical} gives the full proof.

\medskip

\textbf{Notation:} We use $|\cdot|$ for the standard map $\mb{F}_p\to\{0,\ldots,p-1\}$. We often treat $\F_p$ as an additive subgroup of $\R/\Z$ via the map $x\mapsto |x|/p$ and, by some abuse of notation, freely switch between these two viewpoints. We use $e\colon\R/\Z\to\C$ for the function $e(x)=e^{2\pi i x}$ and $e_p\colon\F_p\to\C$ for the function $e_p(x)=e^{2\pi i |x|/p}$.

\section{Background on non-classical polynomials}
\label{sec:background}

\begin{definition}
Fix a prime $p$, a finite-dimensional $\mb{F}_p$-vector space $V$, and an abelian group $G$. Given a function $f\colon V\to G$ and a shift $h\in V$, define the \textbf{additive derivative} $\Delta_hf\colon V\to G$ by \[(\Delta_hf)(x)=f(x+h)-f(x).\]
Given a function $f\colon V\to\mb{C}$ and a shift $h\in V$, define the \textbf{multiplicative derivative} $\partial_hf\colon V\to\mb{C}$ by \[(\partial_hf)(x)=f(x+h)\overline{f(x)}.\]
\end{definition}

\begin{definition}
Fix a prime $p$ and a finite-dimensional $\mb{F}_p$-vector space $V$. Given a function $f\colon V\to\mb{C}$ and $d\ge 1$, the {\textbf{Gowers uniformity norm}} $\|f\|_{U^d}$ is defined by \[\|f\|_{U^d}=|\mb{E}_{x,h_1,\ldots,h_d\in V}(\partial_{h_1}\cdots\partial_{h_d}f)(x)|^{1/2^d}.\]
\end{definition}

See \cite[Lemma B.1]{TZ12} for some basic facts about the Gowers uniformity norms. 

\begin{definition}
Fix a prime $p$, and a non-negative integer $d\ge0$. Let $V$ be a finite-dimensional $\mb{F}_p$-vector space. A \textbf{non-classical polynomial} of degree at most $d$ is a map $P\colon V\to\mb{R}/\mb{Z}$ that satisfies \[(\Delta_{h_1}\cdots \Delta_{h_{d+1}}P)(x)= 0\] for all $h_1,\ldots,h_{d+1},x\in V$.
We write $\on{Poly}_{\leqslant d}(V\to\mb{R}/\mb{Z})$ for the set of non-classical polynomials of degree at most $d$.

A classical polynomial is a map $V\to\mb{F}_p$ satisfying the same property. By composing with the standard map $x\mapsto |x|/p$ we can view $\on{Poly}_{\leqslant d}(V\to\mb{F}_p)$ as a subset of $\on{Poly}_{\leqslant d}(V\to\mb{R}/\mb{Z})$.
\end{definition}

See \cite[Lemma 1.7]{TZ12} for some properties of non-classical polynomials. We give one property below which will be used several times in this paper.

\begin{lemma}[{\cite[Lemma 1.7(iii)]{TZ12}}]
\label{poly-basis}
Fix a prime $p$ and a finite-dimensional $\mb{F}_p$-vector space $V=\mb{F}_p^n$. Then $P\colon V\to\mb{R}/\mb{Z}$ is a non-classical polynomial of degree at most $d$ if and only if it can be expressed in the form
\[P(x_1,\ldots,x_n)=\alpha+\sum_{\genfrac{}{}{0pt}{}{0\le i_1,\ldots,i_n<p,\, j\ge 0:}{0<i_1+\cdots+i_n\le d-j(p-1)}}\frac{c_{i_1,\ldots,i_n,j}|x_1|^{i_1}\cdots|x_n|^{i_n}}{p^{j+1}}\Mod 1,\]
for some $\alpha\in\mb{R}/\mb{Z}$ and coefficients $c_{i_1,\ldots,i_n,j}\in\{0,\ldots,p-1\}$.
Furthermore, this representation is unique.
\end{lemma}

Define $\U_k\subset\R/\Z$ to be $\{0,1/p^k,\ldots,(p^k-1)/p^k\}$.
As a corollary we see that in characteristic $p$, every non-classical polynomial of degree at most $d$ takes values in a coset of $\mb{U}_{\lfloor(d-1)/(p-1)\rfloor+1}$.

Finally we state the inverse theorem of Tao and Ziegler.

\begin{theorem}[{\cite[Theorem 1.10]{TZ12}}]
\label{thm:inverse}
Fix a prime $p$, a positive integer $k$, and a parameter $\delta>0$. There exists $\epsilon>0$ such that the following holds. Let $V$ be a finite-dimensional $\mb{F}_p$-vector space. Given a function $f\colon V\to\mb{C}$ satisfying $\|f\|_\infty\le1$ and $\|f\|_{U^k}>\delta$, there exists a non-classical polynomial $P\in\on{Poly}_{\leqslant k-1}(V\to\mb{R}/\mb{Z})$ such that \[|\mb{E}_{x\in V} f(x)e^{-2\pi i P(x)}|\ge\epsilon.\]
\end{theorem}

Earlier works of Bergelson, Tao, and Ziegler \cite{BTZ10} and Tao and Ziegler \cite{TZ10} show this result in the high-characteristic regime $p\geq k$, with the additional guarantee that $P$ is a classical polynomial of degree at most $k-1$.

\section{Classical polynomials for the \texorpdfstring{$U^{p+1}$}{Up+1}-inverse theorem}
\label{sec:p+1-inverse}

The inverse theorem for the $U^k$-norm does not require non-classical polynomials when $p\geq k$ for the simple reason that every non-classical polynomial of degree at most $p-1$ is a classical polynomial of the same degree (up to a constant shift). To prove \cref{thm:p+1-inverse}, about the $U^{p+1}$-inverse theorem, we use the following fact. Every non-classical polynomial of degree $p$ agrees with a classical polynomial on a codimension 1 hyperplane (up to a constant shift).

\begin{proposition}
\label{thm:p+1-hyperplane}
Let $P\in\on{Poly}_{\leqslant p}(V\to\mb{R}/\mb{Z})$ be a non-classical polynomial of degree at most $p$. Then there exists a codimension 1 hyperplane $U\le V$, a classical polynomial $Q\in\Poly_{\leqslant p}(V\to\F_p)$, and $\alpha\in\R/\Z$ such that $P(x)=\alpha+|Q(x)|/p$ for all $x\in U$.
\end{proposition}

\begin{proof}
Pick an isomorphism $V\simeq \F_p^n$. By \cref{poly-basis}, we have \[P(x_1,\ldots,x_n)=\alpha+P'(x_1,\ldots,x_n)+\frac{c_1|x_1|+\cdots+c_n|x_n|}{p^2}\pmod{1}\] for $\alpha\in\mb{R}/\mb{Z}$, a polynomial $P'$ taking values in $\U_1=\{0,1/p,\ldots,(p-1)/p\}$, and $c_1,\ldots,c_n\in\{0,\ldots,p-1\}$. Define the codimension 1 hyperplane $U\le V$ by $c_1x_1+\cdots+c_nx_n=0$. Note that for $(x_1,\ldots,x_n)\in U$, we have $c_1|x_1|+\cdots+c_n|x_n|\equiv0\Mod p$. Thus $P|_U$ takes values in $\alpha+\U_1$. Thus by our identification of $\U_1$ with $\F_p$,  $P|_U-\alpha$ is a classical polynomial of degree at most $p$.
\end{proof}

\begin{proof}[Proof of \cref{thm:p+1-inverse}]
By the usual inverse theorem, \cref{thm:inverse}, there exists $P\in\on{Poly}_{\leqslant p}(V\to\mb{R}/\mb{Z})$ such that $|\mb{E}_{x\in V}f(x)e(-P(x))|\ge\epsilon$. By \cref{thm:p+1-hyperplane}, there exists a codimension 1 hyperplane $U$ and $\alpha\in\mb{R}/\mb{Z}$ such that $P|_U$ takes values in $\alpha+\U_1$, i.e., $P|_U-\alpha$ is classical. Pick an isomorphism $V\simeq\mb{F}_p^n$ such that $U$ is the hyperplane defined by $x_1=0$. In this basis, there exists a classical polynomial $Q\in\on{Poly}_{\leqslant p}(\mb{F}_p^n\to\mb{F}_p)$ and $c\in\{0,\ldots,p-1\}$ so that \[P(x_1,\ldots,x_n)=\alpha+\frac{|Q(x_1,\ldots,x_n)|}p+\frac{c|x_1|}{p^2}\pmod{1}.\] Thus we have
\begin{align*}
\epsilon
&\le|\mb{E}_{x\in\mb{F}_p^n}f(x)e(-P(x))|\\
&\le\mb{E}_{x_1\in\mb{F}_p}\bigg|\mb{E}_{y\in\mb{F}_p^{n-1}}f(x_1,y)e(-P(x_1,y))\bigg|\\
&=\mb{E}_{x_1\in\mb{F}_p}\bigg|\mb{E}_{y\in\mb{F}_p^{n-1}}f(x_1,y)e_p(-Q(x_1,y))\bigg|\\
&\le\bigg(\mb{E}_{x_1\in\mb{F}_p}\bigg|\mb{E}_{y\in\mb{F}_p^{n-1}}f(x_1,y)e_p(-Q(x_1,y))\bigg|^2\bigg)^{1/2}.
\end{align*}

By Parseval and the pigeonhole principle, there exists $a\in\mb{F}_p$ such that \[\bigg|\mb{E}_{x_1\in\mb{F}_p}e(-ax_1)\mb{E}_{y\in\mb{F}_p^{n-1}}f(x_1,y)e_p(-Q(x_1,y))\bigg|\ge\epsilon/\sqrt{p}.\]

Therefore $f$ has correlation at least $\epsilon/\sqrt{p}$ with the classical polynomial $Q(x_1,\ldots,x_n)+ax_1$.
\end{proof}

\section{Symmetrization techniques}\label{sec:symmetrization}

We now extend a symmetrization technique of Alon and Beigel \cite{AB01} which will be needed to prove the non-correlation property of our example. At a high level, this technique use Ramsey theory to show that if a function correlates with a bounded degree polynomial, then some restriction of coordinates correlates with a symmetric polynomial. Unfortunately, as stated this only holds for multilinear polynomials, and otherwise one can only reduce to the class of so-called quasisymmetric polynomials. These are a generalization of the notion of symmetric polynomials which have found extensive use in enumerative and algebraic combinatorics.

\begin{definition}\label{def:quasisymmetric}
For a prime $p$ and a tuple $(\alpha_1,\ldots,\alpha_s)$ of positive integers satisfying $\alpha_i<p$ for all $i$, the \textbf{elementary quasisymmetric polynomial associated to} $(\alpha_1,\ldots,\alpha_s)$ in $n$ variables is the polynomial $Q_\alpha\colon\F_p^n\to\F_p$ defined by
\[Q_\alpha(x_1,\ldots,x_n) = \sum_{1\le i_1 < \cdots < i_s\le n}\prod_{j=1}^s x_{i_j}^{\alpha_j}.\]
We additionally note the total degree is $|\alpha|:=\alpha_1+\cdots+\alpha_s$.
\end{definition}

\begin{theorem}\label{thm:quasisymmetric-restriction}
Fix a prime $p$ and an integer $d\ge 1$. For any $n$, there exists $m=\omega_{p,d;n\to\infty}(1)$ such that the following holds. Let $P(x_1,\ldots,x_n)$ be a polynomial of degree at most $d$ with coefficients in $\mb{F}_p$. There exists $I\subseteq[n]$ of size $|I|=m$ such that for any $y_{[n]\setminus I}\in\mb{F}_p^{[n]\setminus I}$, the function $P(x_I,y_{[n]\setminus I})$ (viewed as a polynomial in the $x_I$) can be written as a quasisymmetric polynomial of degree $d$ plus an arbitrary polynomial of degree at most $d-1$.
\end{theorem}

\begin{proof}
We can uniquely express $P$ in the form
\[P(x_1,\ldots,x_n)=\sum_{s=0}^d\sum_{1\leq i_1<\cdots<i_s\leq n}\sum_{\substack{\alpha\in\{1,\ldots,p-1\}^s\\|\alpha|\leq d}}c_{i_1,\ldots,i_s,\alpha}\prod_{j=1}^s x_{i_j}^{\alpha_j}\] where the $c_{i_1,\ldots,i_s,\alpha}\in\F_p$ are arbitrary.

Define $\Lambda=\{\lambda\in\{0,\ldots,p-1\}^d:|\lambda|=d\}$. We define a coloring of the complete $d$-uniform hypergraph on $n$ vertices where the set of colors is $\F_p^{\Lambda}$. For an edge $\{i_1,\ldots,i_d\}$ with $1\leq i_1<\cdots<i_d\leq n$, let the color of this edge be given by $c_{i_1,\ldots,i_d}\colon \Lambda\to\F_p$. We define $c_{i_1,\ldots,i_d}(\lambda)=c_{j_1,\ldots,j_s,\alpha}$ where $\alpha$ is formed by removing the 0's from the tuple $\lambda$ and $(j_1,\ldots,j_s)$ is formed from $(i_1,\ldots,i_d)$ by removing the coordinate $i_k$ if $\lambda_k=0$. 

Applying the hypergraph Ramsey theorem, there exists a subset $I\subseteq[n]$ such that the induced subhypergraph on vertex set $I$ is colored monochromatically with color $c\colon\Lambda\to\F_p$ and $|I|=\omega_{p,d;n\to\infty}(1)$. Unwinding the definitions, we see that
\[P(x_{I},y_{[n]\setminus I}) = \sum_{s=0}^d\sum_{\substack{\alpha\in\{1,\ldots,p-1\}^s\\|\alpha|=d}}c(\alpha,\underbrace{0,\ldots,0}_{d-s\text{ 0's}})Q_\alpha(x_I)+ \text{mixed terms}.\]

Now the mixed terms involve at least one factor of $y_{[n]\setminus I}$, so their total $x_I$-degree is strictly smaller than $d$. 
\end{proof}

\section{Non-classical polynomials are necessary}
\label{sec:non-classical}

In this section we prove \cref{thm:main}, namely that non-classical polynomials are necessary in the $U^{k+1}$ inverse theorem when $k> p$. To do this, we use the function $f_n^{(k)}\colon\mb{F}_p^n\to\mb{R}/\mb{Z}$ defined by
\begin{equation}\label{eqn:f_k}
f_n^{(k)}(x) = \frac{1}{p^{\ell+1}}\sum_{i=1}^n|x_i|^r
\end{equation}
where $k=r+(p-1)\ell$ with $\ell\geq 1$ and $0<r<p$. Note that $f_n^{(k)}$ is a non-classical polynomial of degree $k$, so $\|e(f_n^{(k)})\|_{U^{k+1}}=1$.

In order to motivate our proof, suppose for the sake of contradiction that $f_n^{(k)}$ has correlation at least $\epsilon$ with some classical polynomial of degree at most $k$. By \cref{thm:quasisymmetric-restriction}, we will be able to reduce to the situation \[\epsilon\leq\left|\E_{x\sim\F_p^n}e(f_n^{(k)}(x)+g(x)+h(x))\right|\] where $g$ is a homogeneous quasisymmetric polynomial of degree $k$ and $h$ is a classical polynomial of degree at most $k-1$. By the monotonicity of the Gowers norms (alternatively by the Gowers-Cauchy-Schwarz inequality), we deduce
\begin{align*}
\epsilon^{2^k}
&\leq\left|\E_{x}e(f_n^{(k)}(x)+g(x)+h(x))\right|^{2^k}\\
&=\E_{x,h_1,\ldots,h_k}(\partial_{h_1}\cdots\partial_{h_k}e(f_n^{(k)}+g+h))(x)\\
&=\E_{h_1,\ldots,h_k}e(\Delta_{h_1}\cdots\Delta_{h_k}(f_n^{(k)}+g)).
\end{align*}
Since $f_n^{(k)}$ and $g$ are polynomials of degree $k$, the iterated derivatives $(\Delta_{h_1}\cdots\Delta_{h_k}f_n^{(k)})(x)$ and $(\Delta_{h_1}\cdots\Delta_{h_k}g)(x)$ are constants independent of $x$. Furthermore, they take values in $\U_1$ which (with some abuse of notation) we identify with $\F_p$. Many results on these objects are known, including the fact that in general they are multilinear functions of $h_1,\ldots, h_k$ (see \cite[Section 4]{TZ12}). For the purposes of this paper, it is sufficient to do the following explicit computation.

\begin{lemma}
\label{lem:derivative-computation}
For $k=r+(p-1)\ell$ with $\ell\geq 0$ and $0<r<p$,
\[\iota_k(h_1,\ldots,h_k):=\Delta_{h_1}\cdots\Delta_{h_k}f_n^{(k)}=(-1)^\ell r!\sum_{i=1}^n(h_1)_i\cdots(h_k)_i,\]
and for $\alpha=(\alpha_1,\ldots,\alpha_s)$ with $\alpha_1+\cdots+\alpha_s=k$ and $0\leq \alpha_i<p$,
\[\tau_\alpha(h_1,\ldots,h_k):=\Delta_{h_1}\cdots\Delta_{h_k}Q_\alpha=\sum_{\pi\in\mathfrak S_k}\sum_{\vec \imath}(h_1)_{i_{\pi(1)}}\cdots(h_k)_{i_{\pi(k)}}\]where the sum is over sequences $1\leq i_1\leq\cdots\leq i_k\leq n$ that satisfy $i_{\alpha_1+\cdots+\alpha_j+1}=\cdots=i_{\alpha_1+\cdots+\alpha_j+\alpha_{j+1}}$ and $i_{\alpha_1+\cdots+\alpha_j}<i_{\alpha_1+\cdots+\alpha_j+1}$ for all $j$.
\end{lemma}

\begin{proof}[Proof of \cref{lem:derivative-computation}]
For classical polynomials $P,Q\colon\F_p^n\to\F_p$, of degrees $d_1,d_2$, the discrete Leibniz rule, $\Delta_h(PQ)=(\Delta_h P)Q+P(\Delta_h Q)+(\Delta_h P)(\Delta_h Q)$, can be easily verified. This implies the more convenient $\Delta_h(PQ)\equiv (\Delta_h P)Q+P(\Delta_h Q)\pmod{\Poly_{\leqslant d_1+d_2-2}(\F_p^n\to\F_p)}$. Note that taking $d$ discrete derivatives kills $\Poly_{\leqslant d}(\F_p^n\to\F_p)$, so if $P\equiv Q \pmod{\Poly_{\leqslant d}(\F_p^n\to\F_p)}$, then $\Delta_{h_1}\cdots\Delta_{h_d}P(x)=\Delta_{h_1}\cdots\Delta_{h_d}Q(x)$.

We compute $\tau_\alpha$ first. Define $f\in\Poly_{\leqslant k}(V\to\F_p)$ by $f(x)=x_{i_1}\cdots x_{i_k}$. By many applications of the discrete Leibniz rule, we see that \[\Delta_{h_1}\cdots\Delta_{h_k}f(x)=\sum_{\pi\in\mathfrak S_k}\Delta_{h_1}(x_{i_{\pi(1)}})\cdots\Delta_{h_k}(x_{i_{\pi(k)}})=\sum_{\pi\in\mathfrak S_k}(h_1)_{i_{\pi(1)}}\cdots(h_k)_{i_{\pi(k)}}.\]
Extending this result by linearity gives the formula for $\tau_\alpha$.

Now for $\iota_k$. For any $k\geq 1$, write $k=r+(p-1)\ell$ with $0<r<p$ and $\ell\geq 0$. Define $Q_k\in\Poly_{\leqslant k}(\F_p\to\R/\Z)$ by $Q_k(x)=|x|^r/p^{\ell+1}$. By linearity, it suffices to prove that $\Delta_{h_1}\cdots\Delta_{h_k} Q_k(x)=r!(-1)^\ell h_1\cdots h_k$. We prove that $\Delta_hQ_k(x)\equiv r|h|Q_{k-1}(x) \pmod{\Poly_{\leqslant k-2}(\F_p\to\R/\Z)}$ for $k\geq 2$, while obviously $\Delta_hQ_1(x)=|h|/p$. Iterating (and applying the fact that $(p-1)!\equiv -1\pmod p$) gives the desired result.

We break into two cases. First, if $r=1$ (and $\ell\geq 1$) then
\begin{align*}
\Delta_hQ_k(x)
&=\frac{|x+h|-|x|}{p^{\ell+1}}\\
&=\frac{|h|}{p^{\ell+1}}-\frac{\mathbbm{1}(|x|+|h|\geq p)}{p^\ell}\\
&=\frac{|h|}{p^{\ell+1}}-\sum_{c=p-|h|}^{p-1}\frac{\mathbbm{1}(|x|=c)}{p^\ell}\\
&=\frac{|h|}{p^{\ell+1}}+|h|\frac{|x|^{p-1}}{p^\ell}-\sum_{c=p-|h|}^{p-1}\frac{\mathbbm1(|x|=c)+|x|^{p-1}}{p^\ell}.
\end{align*}
Now $\mathbbm1(|x|=c)\equiv 1-(|x|-c)^{p-1}\pmod p$, say $\mathbbm1(|x|=c)=1-(|x|-c)^{p-1}+pE_{c,p}(x)$ for some function $E_{c,p}$. Then we see
\[\Delta_hQ_k(x)
-|h|\frac{|x|^{p-1}}{p^\ell}
=\frac{|h|}{p^{\ell+1}}
-\sum_{c=p-|h|}^{p-1}\frac{1-(|x|-c)^{p-1}+|x|^{p-1}}{p^\ell}
-\sum_{c=p-|h|}^{p-1}\frac{E_{c,p}(x)}{p^{\ell-1}}.\]
We know that every term in this equation is a non-classical polynomial of degree at most $k-1$ except for the last term. Thus we conclude that the last term is also a non-classical polynomial of degree $k-1$. Furthermore, of the three terms on the right-hand side, the first is a constant, the second is a non-classical polynomial of degree at most $(p-1)(\ell-1)+(p-2)=k-2$ (since the $|x|^{p-1}/p^\ell$ terms cancel), and the third has degree at most $(p-1)(\ell-1)=k-p$ (since it takes values in $\U_{\ell-1}$). Thus the right-hand side lies in $\Poly_{\leqslant k-2}(\F_p\to\R/\Z)$, proving the desired result in the $r=1$ case.

Now assume that $k=r+(p-1)\ell$ where $r\geq 2$. We compute 
\begin{align*}
\Delta_h Q_k(x)
&=\frac{|x+h|^r-|x|^r}{p^{\ell+1}}\\
&=\frac{(|x|+|h|)^r-|x|^r}{p^{\ell+1}}-\mathbbm1(|x|+|h|\geq p)\frac{(|x|+|h|)^r-(|x|+|h|-p)^r}{p^{\ell+1}}\\
&=\frac{\sum_{i=1}^r\binom ri |h|^i|x|^{r-i}}{p^{\ell+1}}-\mathbbm1(|x|+|h|\geq p)\left(\sum_{i=1}^r\frac{(-1)^{i-1}\binom ri (|x|+|h|)^{r-i}}{p^{\ell+1-i}}\right).
\end{align*}
We rewrite this as
\[\Delta_h Q_k(x)-r|h|\frac{|x|^{r-1}}{p^{\ell+1}}=\frac{\sum_{i=2}^r\binom ri |h|^i|x|^{r-i}}{p^{\ell+1}}-\mathbbm1(|x|+|h|\geq p)\left(\sum_{i=1}^r\frac{(-1)^{i-1}\binom ri (|x|+|h|)^{r-i}}{p^{\ell+1-i}}\right).\]
We know that every term in this equation is a non-classical polynomial of degree at most $k-1$ except for the last term, implying that the last term is also a non-classical polynomial of degree $k-1$. Furthermore, of the two terms on the right-hand side, the first is a non-classical polynomial of degree at most $(p-1)\ell+r-2=k-2$ and the second has degree at most $(p-1)\ell=k-r$ (since it takes values in $\U_{\ell-1}$). Thus the right-hand side lies in $\Poly_{\leqslant k-2}(\F_p\to\R/\Z)$, proving the desired result in the $r\geq 2$ case.
\end{proof}

Define the maps $I_k,T_\alpha\colon (\F_p^n)^{k-1}\to\F_p^n$ by the equations $\iota_k(h_1,\ldots,h_k)=I_k(h_1,\ldots,h_{k-1})\cdot h_k$ and $\tau_\alpha(h_1,\ldots,h_k)=T_\alpha(h_1,\ldots,h_{k-1})\cdot h_k$. From \cref{lem:derivative-computation}, clearly $I_k(h_1,\ldots,h_{k-1})_i=(-1)^\ell r!(h_1)_i\cdots(h_{k-1})_i$. To continue the argument we will need to show that $T_\alpha$ can be expressed in a particularly convenient form.

\begin{lemma}
\label{lem:multiaffine-form}
For $\alpha=(\alpha_1,\ldots,\alpha_s)$ with $\alpha_1+\cdots+\alpha_s=k$ and $0<\alpha_i<p$ for all $i$,
\[T_\alpha(h_1,\ldots,h_{k-1})_i=\sum_{J\subseteq[k-1]}C_{i,\alpha,J}(h_{[k-1],<i},(\tau_\beta(h_I))_{\beta,I})\prod_{j\in J}(h_j)_i\] for some functions $C_{i,\alpha,J}(\cdot,\cdot)$, evaluated at the tuple of $(h_j)_{i'}$ for all $j\in[k-1]$ and $i'<i$ and the tuple of $\tau_\beta(h_I)$ for all $I\subseteq[k-1]$ and $\beta=(\beta_1,\ldots,\beta_t)$ with $\beta_1+\cdots+\beta_t= |I|$ and $0<\beta_i<p$ for all $i$. Furthermore, \[C_{i,\alpha,[k-1]}=(-1)^{s-1}\alpha_1(k-1)!.\]
\end{lemma}

\begin{proof}
Fix $i\in[n]$ for the rest of the proof. We introduce some notation. We have $h_1,\ldots,h_{k-1}\in\F_p^n$. For $I\subseteq[k-1]$, we write $h_I=(h_j)_{j\in I}$. We use \[h_{I,<}=((h_j)_1,\ldots,(h_j)_{i-1})_{j\in I}\in(\F_p^{i-1})^I\qquad\text{and}\qquad h_{I,>}=((h_j)_{i+1},\ldots,(h_j)_n)_{j\in I}\in(\F_p^{n-i})^I.\]
For $\alpha=(\alpha_1,\ldots,\alpha_s)$ we write $\alpha_{<\ell}=(\alpha_1,\ldots,\alpha_{\ell-1})$ and $\alpha_{>\ell}=(\alpha_{\ell+1},\ldots,\alpha_s)$. We define $\alpha_{\leqslant\ell}$ analogously. Recall that we use $|\alpha|=\alpha_1+\cdots+\alpha_s$. 

By inspection, we can write
\begin{equation}\label{eq:Talpha-initial}
T_{\alpha}(h_{[k-1]})_i=\sum_{\ell=1}^s\alpha_\ell!\sum_{\substack{I\sqcup J\sqcup K=[k-1]: \\ |I|=|\alpha_{<\ell}|,\\ |J|=\alpha_{\ell}-1, \\|K|=|\alpha_{>\ell}|}}\left(\prod_{j\in J}(h_j)_i\right)\tau_{\alpha_{<\ell}}(h_{I,<})\tau_{\alpha_{>\ell}}(h_{K,>}).
\end{equation}

We now remove the terms depending on $h_{[k-1],>}$. Take $\beta=(\beta_1,\ldots,\beta_t)$ with $0< \beta_i<p$ for all $i$ and $L\subseteq[k-1]$ with $|L|=|\beta|$. We have the identity
\begin{equation}\label{eq:PIE}
\begin{split}
\tau_\beta(h_{L,>})=
\tau_\beta(h_L)&-\sum_{\ell=1}^t\sum_{\substack{I\sqcup  K=L: \\ |I|=|\beta_{\leqslant\ell}|, \\|K|=|\beta_{>\ell}|}}\tau_{\beta_{\leqslant\ell}}(h_{I,<})\tau_{\beta_{>\ell}}(h_{K,>})\\
&-\sum_{\ell=1}^t\beta_\ell!\sum_{\substack{I\sqcup J\sqcup K=L: \\ |I|=|\beta_{<\ell}|,\\ |J|=\beta_{\ell}, \\|K|=|\beta_{>\ell}|}}\left(\prod_{j\in J}(h_j)_i\right)\tau_{\beta_{<\ell}}(h_{I,<})\tau_{\beta_{>\ell}}(h_{K,>}).
\end{split}
\end{equation}

Repeatedly applying this identity eventually puts $T_{\alpha}(h_{[k-1]})_i$ into the desired form. To see this, note that applying this identity to $\tau_\beta(h_{L,>})$ produces many terms of the form $\tau_{\beta_{>\ell}}(h_{K,>})$ but all of these satisfy $|\beta_{>\ell}|=|K|<|\beta|=|L|$, so we always make progress. 

Finally we need to compute $C_{i,\alpha,[k-1]}$. Obviously this coefficient is a constant since the final decomposition that we produce is multilinear in the $h_1,\ldots,h_{k-1}$. Furthermore, the only way to produce a term that is a multiple of $(h_1)_i\cdots(h_{k-1})_i$ is to have no factors of $\tau_{\beta_{<\ell}}(h_{I,<})$ in that term. (However, we have a choice of $I,J,K$ in \cref{eq:Talpha-initial}.) This means that in the initial decomposition we need to be in the $\ell=1$ case of the sum and every time we use the identity \cref{eq:PIE} we need to be in the $\ell=1$ case of the second sum. Again, in every subsequent choice although $\ell = 1$ is fixed, we have a choice of $I,J,K$. Tracing through all of these reductions, we see that we produce the coefficient
\[\alpha_1!\binom{k-1}{\alpha_1-1}\prod_{j=2}^s\bigg(-(\alpha_j!)\binom{k-\alpha_1-\cdots-\alpha_{j-1}}{\alpha_j}\bigg) = (-1)^{s-1}\alpha_1(k-1)!.\qedhere\]
\end{proof}

So far we have developed the tools to, starting with the assumption of correlation with a classical polynomial, reduce to a situation in which
\begin{align*}
\epsilon^{2^k}
&\leq\E\left[e_p\left(\iota_k-\sum_\alpha c_\alpha \tau_\alpha\right)\right]=\Pr\left(I_k+\sum_\alpha c_\alpha T_\alpha=0\right)\\
&=\Pr_{h_1,\ldots,h_{k-1}\sim\F_p^n}\left(\forall i\in[n],\: I_k(h_1,\ldots,h_{k-1})_i+\sum_\alpha c_\alpha T_\alpha(h_1,\ldots,h_{k-1})_i=0\right).
\end{align*}
Therefore, we need a bound on the probability that a multiaffine function equals zero.

\begin{lemma}
\label{lem:zero-prob}
Let $L\colon\F_p^r\to\F_p$ be a multiaffine function whose leading coefficient (i.e., coefficient of $x_1\cdots x_r$) is non-zero. Then \[\Pr_{x_1,\ldots,x_r\sim\F_p}(L(x_1,\ldots,x_r)=0)\leq1-\left(1-\frac1p\right)^r=:1-c_{p,r}.\]
\end{lemma}

\begin{proof}
We prove this result by induction on $r$. The bound is trivially true for $r=0$.

For $r\geq 1$, we can write \[L(x_1,\ldots,x_r)=x_rM(x_1,\ldots,x_{r-1})+N(x_1,\ldots,x_{r-1})\] where $M$ and $N$ are multiaffine and the leading coefficient of $M$ is non-zero. Then for each fixed $x_1,\ldots,x_{r-1}$, there is at most 1 choice of $x_r$ that makes $L$ vanish unless $M(x_1,\ldots,x_r)=0$. Then\[\Pr_{x_1,\ldots,x_r\sim\F_p}(L(x_1,\ldots,x_r)\neq 0)\geq\left(1-\frac1p\right)\Pr_{x_1,\ldots,x_{r-1}\sim\F_p}(M(x_1,\ldots,x_{r-1})\neq 0).\]The second term can be handled by the inductive hypothesis.
\end{proof}

We now have the tools to prove the main theorem. The probability we are considering is the probability that $n$ multiaffine functions vanish simultaneously. If these were independent, by the above lemma, we could bound the probability by $(1-c_{p,k})^n=o_{p,k;n\to\infty}(1)$.

In order to introduce such independence, we can take a union bound over all possible $\tau_\beta(h_I)$. Then \cref{lem:multiaffine-form} shows that our multiaffine forms have the following property: if we plug in values for $((h_1)_{i'},\ldots,(h_{k-1})_{i'})_{i' < i}$ and $\tau_\beta(h_I)$ for all $\beta,I$, then $T_\alpha(h_1,\ldots,h_{k-1})_i$ is multiaffine in $(h_1)_i,\ldots,(h_{k-1})_i$ with non-zero leading coefficient. Then we may reveal $((h_1)_i,\ldots,(h_{k-1})_i)$ one-by-one for $i\in[n]$, and find that the total probability is bounded by $(1-c_{p,k})^{n}$. As the number of possible choices in the union bound is $O_{p,k}(1)$ we will be able to prove the desired bound.

\begin{proof}[Proof of \cref{thm:main}]
Take $k\geq p+1$. Consider $f_n^{(k)}\colon\F_p^n\to\R/\Z$ defined in \cref{eqn:f_k}. Since $f_n^{(k)}$ is a non-classical polynomial of degree $k$, we know that $\|e(f_n^{(k)})\|_{U^{k+1}}=1$. For a classical polynomial $P\in\Poly_{\leqslant k}(\F_p^n\to\F_p)$, set $\epsilon=|\E_x e(f_n^{(k)}(x)+|P(x)|/p)|$. We will prove that $\epsilon=o_{p,k;n\to\infty}(1)$.

By \cref{thm:quasisymmetric-restriction}, there exists $m=\omega_{p,k;n\to\infty}(1)$ and a subset $I\subseteq[n]$ such that for all $y_{[n]\setminus I}\in\F_p^{[n]\setminus I}$, \[P(x_I,y_{[n]\setminus I})=Q(x_I)+P_{y_{[n]\setminus I}}(x_I)\]where $Q$ is a homogeneous quasisymmetric polynomial of degree $k$ and $P_{y_{[n]\setminus I}}$ is a polynomial of degree at most $k-1$.

Without loss of generality, assume that $I=[m]$. Then \[\epsilon=|\E_{x\sim\F_p^n} e(f_n^{(k)}(x)+|P(x)|/p)|\leq \E_{y\sim\F_p^{n-m}}|\E_{x\sim\F_p^m}e(f_n^{(k)}(x,y)+|P(x,y)|/p)|.\]
Then by the pigeonhole principle, there exists $y\in\F_p^{n-m}$ such that the inner expectation is at least $\epsilon$. Fix this choice of $y$ for the rest of the proof. Note that $f_n^{(k)}(x,y)=f_m^{(k)}(x)+c_y$ where $c_y=f_{n-m}^{(k)}(y)\in\R/\Z$ is a constant.
Thus we have \[\epsilon\leq|\E_{x\sim\F_p^m}e(f_m^{(k)}(x)+|Q(x)|/p+|P_y(x)|/p+c_y)|.\]

Note that the right-hand side is a $U^1$-norm. Using the monotonicity of the Gowers norms (see, e.g., \cite[Lemma B.1.(ii)]{TZ12}) we deduce
\begin{align*}
\epsilon^{2^k}
&\leq\|e(f_m^{(k)}+|Q|/p+|P_y|/p+c_y)\|_{U^1}^{2^k}\\
&\leq\|e(f_m^{(k)}+|Q|/p+|P_y|/p+c_y)\|_{U^k}^{2^k}\\
&=\E_{x,h_1,\ldots,h_k}\partial_{h_1}\cdots\partial_{h_k}e(f_m^{(k)}+|Q|/p+|P_y|/p+c_y)(x)\\
&=\E_{x,h_1,\ldots,h_k}e(\Delta_{h_1}\cdots\Delta_{h_k}(f_m^{(k)}+|Q|/p+|P_y|/p+c_y)(x)).
\end{align*}

Taking $k$ discrete derivatives kills (non-classical) polynomials of degree at most $k-1$ and turns those of degree $k$ into constants. Thus the final expression is equal to \[\E_{h_1,\ldots,h_k}e(\Delta_{h_1}\cdots\Delta_{h_k}(f_m^{(k)}+|Q|/p)).\]

Since $Q$ is a homogeneous quasisymmetric polynomial of degree $k$, it can be written as $\sum_\alpha c_\alpha Q_\alpha$ where $\alpha$ ranges over all tuples $(\alpha_1,\ldots,\alpha_s)$ with $|\alpha|=k$ and $0<\alpha_i<p$ for all $i$ and the $c_{\alpha}\in\F_p$ are arbitrary coefficients. 

We computed $\Delta_{h_1}\cdots\Delta_{h_k}f_m^{(k)}$ and $\Delta_{h_1}\cdots\Delta_{h_k}Q_{\alpha}$ in \cref{lem:derivative-computation}. These are the $k$-linear forms denoted $\iota_k,\tau_\alpha\colon(\F_p^m)^k\to\F_p$ respectively. Thus so far we have shown that \[\epsilon^{2^k}\leq \E_{h_1,\ldots,h_k}e_p\left(\iota_k(h_1,\ldots,h_k)+\sum_\alpha c_\alpha\tau_\alpha(h_1,\ldots,h_k)\right).\] 

For an arbitrary $k$-linear form $\sigma\colon(\F_p^m)^k\to\F_p$, there is a unique $(k-1)$-linear function $S\colon(\F_p^m)^{k-1}\to\F_p^m$ that satisfies $\sigma(h_1,\ldots,h_k)=S(h_1,\ldots,h_{k-1})\cdot h_k$. Furthermore, we have
\[\E_{h_1,\ldots,h_k}e_p(\sigma(h_1,\ldots,h_k))=\E_{h_1,\ldots,h_k}e_p(S(h_1,\ldots,h_{k-1})\cdot h_k)=\Pr_{h_1,\ldots,h_{k-1}}(S(h_1,\ldots,h_{k-1})=0).\]

From this we conclude\[\epsilon^{2^k}\leq\Pr_{h_1,\ldots,h_{k-1}}\left(\forall i\in[m],\: I_k(h_1,\ldots,h_{k-1})_i+\sum_\alpha c_{\alpha}T_\alpha(h_1,\ldots,h_{k-1})_i=0\right).\]

Recall that $I_k(h_1,\ldots,h_{k-1})_i=(-1)^\ell r!(h_1)_i\cdots (h_{k-1})_i$ where $k=r+(p-1)\ell$ with $\ell\geq 1$ and $0<r<p$. Note that $(-1)^\ell r!\neq 0$ in $\F_p$. Furthermore, \cref{lem:multiaffine-form} states that
\[T_{\alpha}(h_1,\ldots,h_{k-1})_i=\sum_{J\subseteq[k-1]}C_{i,\alpha, J}(h_{[k-1],<i},(\tau_\beta(h_I))_{\beta, I})\prod_{j\in J}(h_j)_i.\]
In other words $T_{\alpha}(h_1,\ldots,h_{k-1})_i$, viewed just as a function of $(h_1)_i,\ldots,(h_{k-1})_i$ is multiaffine with coefficients given by $C_{i,\alpha, J}$. Additionally, \cref{lem:multiaffine-form} also gives the critical fact that the leading coefficient, $C_{i,\alpha,[k-1]}$, is equal to $(-1)^{s-1}\alpha_1(k-1)!$ for all $i$. Since $k\geq p+1$, we have that $C_{i,\alpha,[k-1]}=0$ (recall that the coefficients live in $\F_p$).

This implies that $I_k(h_1,\ldots,h_{k-1})_i+\sum_\alpha c_{\alpha}T_\alpha(h_1,\ldots,h_{k-1})_i$, viewed just as a function of $(h_1)_i,\ldots,(h_{k-1})_i$ is multiaffine with non-zero leading coefficient, say
\[I_k(h_1,\ldots,h_{k-1})_i+\sum_\alpha c_{\alpha}T_\alpha(h_1,\ldots,h_{k-1})_i=\sum_{J\subseteq[k-1]}C_{i,J}(h_{[k-1],<i},(\tau_\beta(h_I))_{\beta,I})\prod_{j\in J}(h_j)_i\]where $C_{i,[k-1]}=(-1)^\ell r!\neq 0$ for all $i$.

By \cref{lem:zero-prob}, if the coefficients are fixed then this function vanishes with probability at most $1-c_{p,k}<1$. To complete the proof, we need to show that we can approximately decouple these events. Formally,
\begin{align*}
\epsilon^{2^k}
&\leq\Pr_{h_1,\ldots,h_{k-1}}\left(\forall i\in[m],\: \sum_{J\subseteq[k-1]}C_{i,J}(h_{[k-1],<i},(\tau_\beta(h_I))_{\beta,I})\prod_{j\in J}(h_j)_i=0\right)\\
&=\sum_{(A_{\beta,I})_{\beta, I}}\Pr_{h_1,\ldots,h_{k-1}}\left(\forall \beta, I,\: \tau_\beta(h_I)=A_{\beta,I} \cap
 \sum_{J\subseteq[k-1]}C_{i,J}(h_{[k-1],<i},(\tau_\beta(h_I))_{\beta,I})\prod_{j\in J}(h_j)_i=0 \right)\\
 &\le \sum_{(A_{\beta,I})_{\beta, I}}\Pr_{h_1,\ldots,h_{k-1}}\left(
 \sum_{J\subseteq[k-1]}C_{i,J}(h_{[k-1],<i},(A_{\beta,I})_{\beta,I})\prod_{j\in J}(h_j)_i=0 \right).\\
\end{align*}
The final replacement simply comes by substituting in the values $A_{\beta,I}$.

Now for each $i\in[m]$, let $E_i$ be the event that $\sum_{J\subseteq[k-1]}C_{i,J}(h_{[k-1],<i},(A_{\beta,I})_{\beta,I})\prod_{j\in J}(h_j)_i=0$. We wish to bound
\[\Pr_{h_1,\ldots,h_{k-1}}\left(E_i
\right.\left|
\forall i'<i,\: E_{i'}\right).\]
Since the event we are conditioning on only depends on $h_{[k-1],<i}$, the conditional distribution of $(h_1)_i,\ldots,(h_{k-1})_i$ is still uniform. Thus we can upper bound the above probability by
\begin{align*}
\sup_{h_{[k-1],<i}}\Pr_{(h_1)_i,\ldots,(h_{k-1})_i\sim\F_p}\left(\sum_{J\subseteq[k-1]}C_{i,J}(h_{[k-1],<i},(A_{\beta,I})_{\beta,I})\prod_{j\in J}(h_j)_i=0\right).
\end{align*}
By \cref{lem:zero-prob}, and the fact that $C_{i,[k-1]}=(-1)^\ell r!\neq 0$ always, this probability is upper-bounded by $1-c_{p,k}<1$. Putting everything together, we have shown that $\epsilon^{2^k}\leq O_{p,k}((1-c_{p,k})^{m})$. (The hidden constant is the number of terms in the sum over $(A_{\beta,I})_{\beta, I}$, which depends on $p,k$ but not on $m,n$. It can be bounded by $p^{4^k}$.) We showed that $m=\omega_{p,k;n\to\infty}(1)$, implying that $\epsilon=o_{p,k;n\to\infty}(1)$.
\end{proof}

%\bibliographystyle{amsplain0.bst}
%\bibliography{main.bib}

\providecommand{\bysame}{\leavevmode\hbox to3em{\hrulefill}\thinspace}
\providecommand{\MR}{\relax\ifhmode\unskip\space\fi MR }
\providecommand{\MRhref}[2]{%
  \href{http://www.ams.org/mathscinet-getitem?mr=#1}{#2}
}
\providecommand{\href}[2]{#2}

\end{document}